\documentclass{article}

\usepackage{amsthm,amsmath,amsfonts,amssymb}
\usepackage{latexsym}
\newtheorem{theorem}{Theorem}

\newtheorem{lemma}{Lemma}

\newcommand{\tilo}{\mathrm{q}}

\renewcommand{\u}{\mathbf{u}}
\renewcommand{\v}{\mathbf{v}}
\renewcommand{\d}{\mathrm{d}}
\newcommand{\y}{\mathbf{y}}
\newcommand{\x}{\mathbf{x}}
\newcommand{\z}{\mathbf{z}}
\newcommand{\real}{\mathbb{R}}
\newcommand{\loc}{\scriptsize{loc}}
\newcommand{\vare}{\varepsilon}
\newcommand{\ds}{\displaystyle}

\DeclareMathOperator{\dv}{div}
\DeclareMathOperator{\curl}{curl}

\usepackage{color}

\title{Confinement of vorticity for the 2D Euler-$\alpha$ equations}

\author{David M. Ambrose \and Milton C. Lopes Filho \and Helena J. Nussenzveig Lopes}

\AtEndDocument{\bigskip{\footnotesize%
  \addvspace{\medskipamount}
  \textsc{Department of Mathematics, Drexel University, 3141 Chestnut Street, Philadelphia, PA 19104, USA} \par
  \addvspace{\medskipamount}
  \textit{E-mail address}, D.~M. Ambrose: \texttt{dma68@drexel.edu} \par
  \addvspace{\medskipamount}
  \textsc{Instituto de Matem\'atica, Universidade Federal do Rio de Janeiro, Caixa Postal 68530, 21941-909, Rio de Janeiro, RJ, Brazil} \par
  \addvspace{\medskipamount}
  \textit{E-mail address}, M.~C. Lopes Filho: \texttt{mlopes@im.ufrj.br} \par
  \textit{E-mail address}, H.~J. Nussenzveig Lopes: \texttt{hlopes@im.ufrj.br} \par
}}

\begin{document}

\maketitle

\begin{abstract}
In this article we consider weak solutions of the Euler-$\alpha$ equations in the full plane. We take, as initial unfiltered vorticity, an arbitrary nonnegative, compactly supported, bounded Radon measure. Global well-posedness for the corresponding initial value problem is due M. Oliver and S. Shkoller.  We show that, for all time, the support of the unfiltered vorticity is contained in a disk whose radius grows no faster than $\mathcal{O}((t\log t)^{1/4})$. This result is an adaptation of the corresponding result
for the incompressible 2D Euler equations with initial vorticity compactly supported, nonnegative, and $p$-th power integrable, $p>2$, due to D. Iftimie, T. Sideris and P. Gamblin and, independently, to Ph. Serfati.
\end{abstract}

\section{Introduction}

In this work we study the initial-value problem for the two-dimensional incompressible Euler-$\alpha$ equations, $\alpha > 0$ fixed, given by:
\begin{equation}\label{alphaEulerVel}
\left\{
\begin{array}{ll}
\partial_t \v + \u \cdot \nabla \v + \sum_{j=1}^2 v_j \nabla u_j = -\nabla p, & \mbox{ in } \real^2 \times \real_+,\\
\dv \u = 0,  & \mbox{ in } \real^2 \times \real_+,\\
\v= \u - \alpha^2 \Delta \u,  & \mbox{ in } \real^2 \times \real_+,\\
\u(\cdot,0)=\u_0,  & \mbox{ at } \real^2 \times \{0\}.
\end{array}
\right.
\end{equation}
Above $\u$ is called the {\it filtered} velocity while $\v$ is the {\it unfiltered} velocity. Taking the curl of these
equations yields the vorticity formulation which, in two space dimensions, becomes:

\begin{equation} \label{the-equation}
\left\{
\begin{array}{ll}
\partial_{t}\tilo + \u \cdot\nabla\tilo = 0, & \mbox{ in } \real^2 \times \real_+ \\
\tilo = \curl \v, & \mbox{ in } \real^2 \times \real_+\\
\curl \u=(\mathbb{I}-\alpha^2\Delta)^{-1}\, \tilo \equiv \omega, & \mbox{ in } \real^2 \times \real_+,\\
\dv \u=0, & \mbox{ in } \real^2 \times \real_+,\\
\tilo(\cdot,0)=\tilo_0,  & \mbox{ at } \real^2 \times \{0\}.
\end{array}
\right.
\end{equation}
This is an instance of an active scalar transport equation (see \cite{peterstuff}).

We call $\omega$ and $\tilo$ filtered and unfiltered vorticity, respectively and we consider initial data $\u_0$ such that $\tilo_0 \in \mathcal{BM}_{c,+}(\real^2)$. Global well-posedness for \eqref{the-equation} with this initial data was established by Oliver and Shkoller in \cite{OS2001}.

For $\tilo \in \mathcal{BM}(\real^2)$, we refer to the integrals of the quantities $\tilo$, $\x \tilo$ and $|\x|^2 \tilo$ in the full plane as mass, center of mass and moment of inertia of $\tilo$, respectively.

The purpose of this article is to prove the following result.

\begin{theorem} \label{mainthm}
Let $\tilo_0$ be a nonnegative, compactly supported bounded Radon measure. Let $R_0>0$ be the radius of a disk, centered at the origin, which contains the support of $\tilo_0$. Set $\tilo=\tilo(\x,t) \in L^\infty(0,T;\mathcal{BM}_{+,c}(\real^2))$ to be the unique global weak solution of \eqref{the-equation},  with $\tilo(\x,0)=\tilo_0$. Then, there exist $C>0$, depending on $R_0$, on the mass, center of mass and moment of inertia of $\tilo_0$, such that, for every $t \geq 0$, the support of $\tilo$ is contained in the disk, centered at the origin, with radius
\[R(t) = 8R_0 + C[t(\log (2+t))]^{1/4}.\]
\end{theorem}

This result is an analogue of a corresponding result for the incompressible 2D Euler equations, due to Iftimie, Sideris and Gamblin \cite{ISG1999}, and independently to Serfati \cite{Serfati1998}, valid for vorticity in $L^p_{c,+}$, $p>2$. The proof of Theorem \ref{mainthm} is an adaptation of the proof of Iftimie {\it et alli}.

The Euler-$\alpha$ equations are a regularization of the incompressible Euler equations.
These equations were introduced in the context of averaged fluid models, see \cite{HolmEtAl1998a,HolmEtAl1998b}; if
$\alpha$ is a length scale then the equations describe the flow of an incompressible fluid on spatial scales larger than $\alpha$. The Euler-$\alpha$ equations are the inviscid version of the Lagrangian-averaged Navier-Stokes equations (also known as LANS-$\alpha$), a small-scale closure model for turbulence, and also the inviscid case of the non-Newtonian fluids known as second grade fluids. In addition, the Euler-$\alpha$ equations are an infinite dimensional Hamiltonian system, describing the geodesics in the group of area-preserving diffeomophisms, endowed with a right-invariant metric induced by a parametrized $H^1$-norm on the corresponding Lie algebra, see \cite{MRS2000}.

Confinement, as we study here, is a restriction on the large-time behavior of this system which is closely related to its Hamiltonian nature. It consists of estimating the growth of support of the unfiltered vorticity in terms of some of the conserved quantities associated with the symmetries of the problem.
For the 2D Euler equations the problem of confinement of the support of vorticity was first addressed by C. Marchioro (see \cite{Marchioro1994}), who showed that, for flow in the full plane with bounded, nonnegative vorticity, the diameter of the support cannot grow faster than $\mathcal{O}(t^{1/3})$. (This estimate has been extended to nonnegative $p^{th}$-power integrable vorticities, $p>2$, in \cite{L21998}.) Marchioro's estimate has been  improved to growth no faster than $\mathcal{O}(t^{1/4}\log t)$, see \cite{ISG1999}; see also
\cite{Serfati1998}, where an estimate of the form $\mathcal{O}(t^{1/4}\log \circ \ldots \circ \log t)$ was established. As we will see, besides its Hamiltonian nature, the Euler-$\alpha$ system has conserved quantities which parallel those of the Euler equations, including the first three moments of vorticity and the Hamiltonian itself. The confinement problem has been studied, additionally, for a family of modified surface quasi-geostrophic equations, see \cite{Garra2016}. The technique used was a different adaptation of the work of \cite{ISG1999}.

The relationship between confinement and the conserved quantities is a subtle one. For instance, the upgrade from Marchioro's cubic root estimate to near fourth root followed from the use of the conservation of the center of vorticity, in addition to the total mass and the moment of inertia which were used originally by Marchioro. Another instance illustrating how subtle confinement may be was showcased in the study of confinement in exterior domains, see \cite{Marchioro1996,ILN07}. In a general exterior domain, the best known confinement results yield square root growth in time of the diameter of the support of vorticity. This is due to the loss of symmetry, which implies loss of both the conservation of the center of mass and of the moment of inertia. The estimate is improved to cubic root growth in the special case of the disk, because the moment of inertia is then conserved, recovering fourth root only for even initial vorticity in the exterior of a disk, where both the center of vorticity and the moment of inertia are conserved. In the present work, we extend the fourth root confinement estimate to the Euler-$\alpha$ equations. It was by no means a foregone conclusion whether we would be able to recover this estimate, to improve on it, or to lose it partially or completely. Investigating this issue was the main motivation behind the present work.

The Euler-$\alpha$ equations are an active scalar transport system, where the unfiltered vorticity is transported by a velocity generated by a regularized Biot-Savart law. The confinement analysis for 2D Euler was restricted to vorticities in $L^p$, $p>2$, which is the condition for the velocity to be bounded. For the Euler-$\alpha$ equations, the transporting velocity is bounded, even if the unfiltered vorticity is a bounded measure. We note that our confinement analysis extends to the bounded measure case.

Transport of vorticity by a desingularized velocity is the basic idea behind vortex blob methods, a well-known numerical method for computational modeling of 2D flows. The Euler-$\alpha$ equations can be seen as a specific, somewhat impractical instance of vortex blob, of particular interest due to its geometric significance. We provide conditions under which the confinement estimate obtained for Euler-$\alpha$ can be extended to vortex blob approximations.

Finally, a remark concerning the significance of confinement estimates. An estimate that applies exclusively to flows with distinguished sign vorticities is of limited physical interest. Without that assumption, it is easy to construct examples where the diameter of the support of vorticity grows linearly in time, using vortex dipoles. However, in \cite{MP93}, Marchioro and Pulvirenti used the idea behind confinement estimates to establish the asymptotic validity of the point vortex model as an approximation for the dynamics of flows with sharply concentrated vorticity, a question with definite physical relevance. This was extended to the modified surface quasi-geostrophic equations in \cite{CGM2013}.

The remainder of this paper is organized as follows.
In the next section we will recall basic facts about Bessel potentials, which will be needed further ahead. In the third section we discuss conserved quantities for the Euler-$\alpha$ system of equations and we derive basic \textit{a priori} estimates. The fourth section contains the proof of our confinement result. In the final section we make some concluding remarks and discuss a couple of open problems.

In what follows we write $\x^{\perp}$ to mean $\x^{\perp} = (a,b)^{\perp}=(-b,a)$.

\section{Bessel potentials}

In this article, we require detailed knowledge on the operator $(\mathbb{I}-\alpha^2 \Delta)^{-1}$, by virtue of the relation between filtered and unfiltered vorticity. This is a multiplier operator and, therefore, it can be written as convolution against a kernel, which we denote by $G_{\alpha}=G_\alpha(\x)$. The kernel $G_{\alpha}$ is, in turn, related to the classical Bessel potential $\mathcal{J}_2=\mathcal{J}_2(\x)$, which is the kernel associated to the operator $(\mathbb{I} - \Delta)^{-1}$; in fact it is an easy calculation to show that
\begin{equation} \label{GalphaJ2}
G_{\alpha}(\x)=\frac{1}{\alpha^{2}} \mathcal{J}_2\left(\frac{\x}{\alpha}\right).
\end{equation}
We enumerate the relevant properties of $G_\alpha$, which can be easily derived from the corresponding properties of $\mathcal{J}_2$.
Fix $\alpha > 0$. In view of \eqref{GalphaJ2} the following hold true:
\begin{enumerate}
\item[P1)] $G_{\alpha}$ is radially symmetric; i.e. there exists $g_\alpha = g_\alpha(r)$ such that $G_\alpha(\x)=g_\alpha(|\x|)$;
\item[P2)] $G_{\alpha}$ is positive and has unit integral, independent of $\alpha$. This means that $g_\alpha > 0$ and
\begin{equation} \label{galphaposandint1}
 2\pi \int_0^\infty sg_\alpha (s) \, ds = 1;
\end{equation}
\item[P3)] $G_{\alpha}$ is exponentially decaying as $|\x|\to \infty$. More precisely, for any  $M>0$, there exist positive constants
$c_{1}$ and $c_{2}$ such that
\begin{equation} \label{Galphaexpdecay}
|g_{\alpha}(|\x|)| \leq  c_{1} e^{-c_{2}|\x|}, \mbox{   whenever } \ |\x|>M;
\end{equation}
\item[P4)] $G_\alpha$ has a logarithmic singularity at $0$, i.e., there exist $c_3$, $c_4 >0$ such that, if $|\x|<1$ then
\begin{equation} \label{Galphalogat0}
c_{3} \log \left(\frac{1}{|\x|}\right) \leq |g_{\alpha}(|\x|)| \leq  c_{4} \log \left(\frac{1}{|\x|}\right).
\end{equation}
\end{enumerate}
We refer the reader to Chapter V.3.1 of \cite{Stein1970} for details about P1, P2, and P3, and Chapter 1, expression (1.2.23) of \cite{AH1999} for P4.

Next, we recall the Biot-Savart law, which expresses a divergence-free velocity in terms of its curl:
\[\u = \u(\x) = K \ast \omega (\x) = \int \frac{(\x - \y)^{\perp}}{2\pi|\x-\y|^2}\omega(\y)\,\d\y.\]
Above, $K = K(\x) = \x^\perp / (2\pi|\x|^2)$ is the convolution kernel associated to the multiplier operator $\nabla^{\perp}\Delta^{-1}$.

Since $\omega = G_\alpha \ast \tilo$, we know that $\u = K \ast G_\alpha \ast \tilo$. Set $K_\alpha = K \ast G_\alpha$
and let us seek a representation formula for $K_\alpha$. Note that $K_\alpha$ is the velocity due to a vorticity
given by $G_\alpha(\cdot)$, or, equivalently, due to the radially symmetric vorticity  $g_\alpha (|\cdot|)$. Hence,
the corresponding Biot-Savart law takes the form:
\[K_\alpha (\x) = \frac{\x^\perp}{|\x|^2}\int_0^{|\x|}s g_\alpha(s)\,ds;\]
see Chapter 2, expression (2.14) in \cite{BM2002} for a proof of this fact.

Set
\[\gamma_\alpha = \gamma_\alpha (r) \equiv 2\pi \int_0^r s g_\alpha (s) \, ds.\]
Clearly, for $r>0$,
\begin{equation} \label{gammaalphabdd}
0 < \gamma_\alpha \leq 1.
\end{equation}

We have obtained that
\begin{equation} \label{Kalpha}
K_\alpha = K_\alpha(\x) = K(\x)\gamma_\alpha(|\x|).
\end{equation}
and, hence,
\begin{equation} \label{UintermsofTILO}
\u=\u(\x)= K_\alpha \ast \tilo (\x) = \int K(\x-\y)\gamma_\alpha(|\x-\y|)\tilo(\y)\,\d\y.
\end{equation}

Some of the properties we require concerning $g_{\alpha}$ can be conveniently expressed in terms of $\gamma_{\alpha}$.
We collect those properties in the following result.

\begin{lemma} \label{gammaalphaprops}
We have:
\begin{enumerate}
\item $r\mapsto \displaystyle{\frac{1}{r}}\gamma_\alpha (r)$ is bounded on $[0,\infty)$;
\item there exist positive constants $c$, $\widetilde{c}$, such that $|1-\gamma_\alpha (r)| \leq c e^{-\widetilde{c}r}$, for every $r>0$;
\item $r\mapsto r(1-\gamma_\alpha (r))$ is bounded on $[0,\infty)$.
\end{enumerate}
\end{lemma}

\begin{proof}
Item 1 is a consequence of property P4. Item 2 follows from property P3 and item 3 is a consequence of item 2.
\end{proof}


\section{Conserved Quantities and some {\it a priori} estimates}

We use the material discussed in section 2 to obtain some useful conserved quantities and {\it a priori} estimates for Euler-$\alpha$ dynamics.
The conserved quantities involve low moments of vorticity, and, for convenience, we describe these quantities as the mass, center of mass and moment of inertia for both unfiltered and filtered vorticity. We recall the antisymmetry of $K$, i.e., $K(-\z)=-K(\z)$.

\begin{lemma} \label{consqtties}
Let $\tilo$ be a smooth solution of \eqref{the-equation},  and recall $\u = K_\alpha \ast \tilo $. Assume that $\tilo(\cdot, t)$ is compactly supported, for each $t>0$. Then
\begin{eqnarray}
\label{totmass}
\int \tilo(\x,t)\,\d\x = \int \tilo(\x,0)\,\d\x \equiv m_0; \\
\label{centermass}
\int \x \tilo(\x,t)\,\d\x = \int \x\tilo(\x,0)\,\d\x\equiv Z_0; \\
\label{mominertia}
\int |\x|^2\,\tilo(\x,t)\,\d\x = \int |\x|^2\,\tilo(\x,0)\,\d\x \equiv i_0.
\end{eqnarray}
\end{lemma}

\begin{proof}
The conservation of mass expressed in \eqref{totmass} follows immediately by integrating \eqref{the-equation} in space.

To see that the center of mass is conserved multiply \eqref{the-equation} by $\x$ and then integrate in space. We find, after integrating by
parts and using \eqref{UintermsofTILO},
\[\frac{d}{dt}\int \x \tilo (\x,t) \, \d\x = \int (\u \tilo)(\x,t) \, \d\x \]
\[
= \int\int \frac{(\x - \y)^{\perp}}{2\pi|\x-\y|^2}\gamma_\alpha(|\x-\y|)\tilo(\y,t)\tilo(\x,t)  \,\d\x \d\y= 0,\]
where the vanishing of the last integral can be seen by exchanging $\x$ and $\y$ and using the antisymmetry of the kernel $K$. This implies \eqref{centermass}.

Finally, to show that the moment of inertia is conserved we multiply \eqref{the-equation} by $|\x|^2$ and integrate in space to find, after integrating by parts and inserting \eqref{UintermsofTILO},
\[
\frac{d}{dt}\int |\x|^2\tilo(\x,t)\,\d\x =  2 \int \x \cdot (\u \tilo )(\x,t)\,\d\x \]
\[
= 2\int \x \cdot \int \frac{(\x - \y)^{\perp}}{2\pi|\x-\y|^2}\gamma_\alpha(|\x-\y|)\tilo(\y,t)\tilo(\x,t)  \,\d\x \d\y \]
\[= \int\int (\x-\y)\cdot  \frac{(\x - \y)^{\perp}}{2\pi|\x-\y|^2}\gamma_\alpha(|\x-\y|)\tilo(\y,t)\tilo(\x,t)  \,\d\x \d\y =0,
\]
where we used the antisymmetry of $K$ together with the fact that $\z \cdot K(\z) =0$. We have, thus, established \eqref{mominertia}. This concludes the proof.

\end{proof}

Under the assumptions of Lemma \ref{consqtties}, mass, center of mass and moment of inertia are conserved for the filtered vorticity as well. Indeed, the filtered vorticity is obtained from the unfiltered vorticity by $\omega=G_\alpha\ast\tilo$. It follows from the property P3 and the assumption that $\tilo$ has compact support that $\omega$ is exponentially decaying at infinity.
Thus, from Lemma \ref{consqtties} above together with the relation $\tilo = \omega - \alpha^2 \Delta \omega$, we have:
\begin{equation} \label{massfilt}
\int\omega(\x,t)\,\d\x = \int \tilo(\x,0)\,\d\x = m_0;
\end{equation}
\begin{equation} \label{centmassfilt}
\int \x\omega(\x,t)\,\d\x = \int \x\tilo (\x,0)\,\d\x = Z_0;
\end{equation}
\begin{equation} \label{mominertfilt}
\int |\x|^2\,\omega(\x,t)\,\d\x = \int |\x|^2\,\tilo (\x,0)\,\d\x + 4\alpha^2\int \tilo (\x,0)\,\d\x = i_0 + 4\alpha^2 m_0.
\end{equation}

Next we focus on {\it a priori} estimates, beginning with the bound on the supremum of the filtered velocity.

Fix $T>0$ and assume that $\tilo_0 $ is smooth and compactly supported. We will show that the transporting velocity $\u = K_\alpha \ast \tilo$ is bounded, uniformly in time and, therefore, the support of $\tilo$ remains compact, in space, at any time $t>0$. Recall the inequality
\begin{equation} \label{Dragosestimate}
\|K\ast f\|_{L^\infty(\real^2)} \leq \|f\|_{L^1(\real^2)}^{1-p'/2}\|f\|_{L^p(\real^2)}^{p'/2},
\end{equation}
for any $p>2$, where $p'=p/(p-1)$.
We will use \eqref{Dragosestimate} with $f = G_\alpha \ast \tilo$. We will show that
$\|G_\alpha \ast \tilo\|_{L^1(\real^2)}$ and $\|G_\alpha \ast \tilo\|_{L^p(\real^2)}$, for some $p>2$, are bounded by $\|\tilo\|_{L^1(\real^2)}$.

It is easy to see that
\[\|G_\alpha \ast \tilo\|_{L^1(\real^2)} = \|\tilo\|_{L^1(\real^2)},\]
as $\|G_\alpha\|_{L^1(\real^2)}=1$.

To estimate $\|G_\alpha \ast \tilo\|_{L^p(\real^2)}$, for some $p>2$, we start by recalling that $L^1(\real^2) \subset \mathcal{BM}(\real^2) \subset W^{-s,r}(\real^2)$, for any $s>0$, $r>1$, such that $sr > 2(r-1)$. Since $G_\alpha$ gains two derivatives, see \cite{Stein1970}, it follows that $G_\alpha \ast \tilo \in W^{2-s,r}(\real^2)$, for the same range of $s>0$, $r>1$. We choose $s=1$ and any $r<2$. It follows that $G_\alpha \ast \tilo \in W^{1,r}(\real^2)$ for any $1<r<2$ and, moreover, there exists $C=C_r>0$ such that
\[\|G_\alpha \ast \tilo \|_{W^{1,r}(\real^2)} \leq C_r\|\tilo\|_{L^1(\real^2)}.\]
Finally, we can use the Gagliardo-Nirenberg inequality to conclude that
\[\|G_\alpha \ast \tilo \|_{L^p(\real^2)} \leq C_p\|\tilo\|_{L^1(\real^2)},\]
for $p=r^\ast = 2r/(2-r).$ We note that, since $r<2$, we have $p>2$. Thus, using \eqref{Dragosestimate}, we find
\begin{equation} \label{ubounded}
\|\u\|_{L^\infty(\real^2)} \leq C \|\tilo\|_{L^1(\real^2)}.
\end{equation}
In particular, this completes the proof of (3.4) in \cite{OS2001}.

Next, we  obtain an estimate of the integral of $\omega$ near infinity in terms of the corresponding integral of $\tilo$. For the remainder of this paper we assume that $\tilo$ is nonnegative and compactly supported, conditions which we now know are preserved by the evolution.

We are interested in the quantity $\int_{|\y|>r}\omega(\y,t)\ \d\y,$ for a given $r>0,$ and its relationship to $\int_{|\y|>r'}\tilo(\y,t)\ \d\y$, for some $r'=r'(r)$.
We will use the relation between these two quantities only for $r$ bounded away from zero; hence let us assume, in what follows, that $r>1$.

We investigate this by using the equation $\omega = G_\alpha \ast \tilo$ and properties P2 and P3:
\[\int_{|\y|>r} \omega(\y,t)\,\d\y = \int_{|\y|>r}\int G_\alpha(\y-\x)\tilo(\x,t)\,\d\x \d\y \]
\[=\int_{|\y|>r}\int_{|\x|\leq r/2} g_\alpha(|\x-\y|)\tilo(\x,t)\,\d\x \d\y + \int_{|\y|>r}\int_{|\x|>r/2} G_\alpha(\y-\x)\tilo(\x,t)\,\d\x \d\y\]
\[\leq c_1 \int_{|\y|>r}\int_{|\x|\leq r/2} e^{-c_2|\x-\y|}\tilo(\x,t)\,\d\x \d\y +
\int_{|\x|>r/2}\tilo(\x,t) \int_{|\y|>r} G_\alpha(\y-\x)\,\d\y \d\x\]
\[\leq c_1 e^{c_2 r/2}\int_{|\x|\leq r/2} \tilo(\x,t)\,\d\x \int_{|\y|>r} e^{-c_2|\y|}\,\d\y +
\int_{|\x|>r/2}\tilo(\x,t) \int_{\real^2} G_\alpha(\y-\x)\,\d\y \d\x \]
\[\leq 2\pi c_1 \left(\frac{1+c_2r}{c_2^2}\right)e^{-c_2r/2}\|\tilo_0\|_{\mathcal{BM}(\real^2)} + \int_{|\x|>r/2}\tilo(\x,t) \,\d\x.\]
We used \eqref{totmass} in the last inequality.

We have hence established that, for some $c>0$,
\begin{equation} \label{filtINTERMSofunfilt}
\displaystyle{\int_{|\y|>r}\omega(\y,t)\, \d\y}  \leq c(1+r)\|\tilo_0\|_{\mathcal{BM}(\real^2)}e^{-cr} +
\displaystyle{\int_{ |\y| > \frac{r}{2}}\tilo(\y,t) \, \d\y}.
\end{equation}

Therefore, the integral of $\tilo$ near infinity controls that of $\omega$, up to an exponentially small error.

Our focus, in this work, is on unfiltered vorticities which are nonnegative, bounded, Radon measures, with compact support. Let $\tilo_0 \in \mathcal{BM}_{+,c}(\real^2)$. Choose a sequence of smooth, nonnegative, compactly supported functions $\tilo_0^n$ converging strongly in $\mathcal{BM}(\real^2)$ to $\tilo_0$. Let $\tilo^n$ be the exact, smooth, solution of the Euler-$\alpha$ equations \eqref{the-equation} with initial unfiltered vorticity $\tilo_0^n$. Then, for $\tilo^n$ we have:
\begin{itemize}
\item $\|\tilo^n(\cdot,t)\|_{L^1(\real^2)} \leq \|\tilo_0\|_{\mathcal{BM}(\real^2)}$;
\item $\u^n=K_\alpha\ast\tilo^n$ is uniformly bounded, with respect to $n$, in $L^\infty((0,T)\times\real^2)$, and
\[\|\u^n(\cdot,t)\|_{L^{\infty}(\real^2)} \leq C\|\tilo_0\|_{\mathcal{BM}(\real^2)};\]
\item $\tilo^n(\cdot,t)$ is compactly supported, for any $t>0$.
\end{itemize}
A standard compactness argument, using that $\u^n$ is bounded in $W^{1-s,r}_{\loc}(\real^2)$ for $s>0$, $r>1$ such that $sr>2(r-1)$, gives
that $\tilo^n \rightharpoonup \tilo$ weak-$\ast$ $L^\infty(0,T;\mathcal{BM}(\real^2))$, with $\tilo$ being the unique weak solution of \eqref{the-equation} with initial data $\tilo_0$, see also \cite{OS2001} for a different approach.

We note that, although relations \eqref{totmass}, \eqref{centermass}, \eqref{mominertia}, \eqref{massfilt}, \eqref{centmassfilt}, \eqref{mominertfilt} were all derived assuming smoothness of $\tilo_0$, it follows by the passage to the limit argument above that these relations hold true for $\tilo_0\in\mathcal{BM}_{+,c}(\real^2)$. Similarly, since the $L^\infty$ estimate for $\u$ depends only on the $L^1$ bound on $\tilo$, which can be substituted by a bound in $\mathcal{BM}$, it follows that $\u$ is bounded even if $\tilo_0$ is only in $\mathcal{BM}(\real^2)$. Finally, similar arguments can be used to extend \eqref{filtINTERMSofunfilt} to $\tilo_0 \in \mathcal{BM}(\real^2)$ as well.

\section{Confinement of vorticity}

Let $\tilo_0$ be a nonnegative, compactly supported, bounded Radon measure. We will show that a weak solution of the Euler-$\alpha$ equations \eqref{the-equation}, having $\tilo_0$ as initial unfiltered vorticity, has support growing no faster than $t^{1/4}|\log t|$.
This fact was proved for weak solutions of the incompressible 2D Euler equations with bounded, compactly supported non-negative vorticity by
D. Iftimie, T. Sideris and P. Gamblin in \cite{ISG1999}. Our proof is an adaptation of their proof; we will often refer to their notation and some of the estimates they have derived, and we will follow their strategy.

As mentioned above, existence and uniqueness of a global-in-time weak solution of the Euler-$\alpha$ equations with $\tilo_0 \in \mathcal{BM}(\real^2)$ was established by M. Oliver and S. Shkoller in \cite{OS2001}. Let $\tilo=\tilo(\x,t)$ denote the unique, global weak solution with initial data $\tilo_0$; from Section 3 we know $\tilo \in L^\infty(0,T;\mathcal{BM}_{+,c}(\real^2))$ .

We are now ready to prove our main result.

\begin{proof}[Proof of Theorem \ref{mainthm}]
The proof consists in an adaptation of the proof of Theorem 2.1 in \cite{ISG1999}. We will begin by explaining the structure of the proof of Theorem 2.1 of \cite{ISG1999}. The idea is to estimate the radial component of velocity; in our case this means showing that the filtered velocity satisfies
\[\left| \u(\x,t) \cdot \frac{\x}{|\x|} \right| \leq \frac{C}{|\x|^3},\]
if $|\x|\geq R(t) \equiv 8R_0 + C[t(\log (2+t)]^{1/4}$.
As in \cite{ISG1999}, this implies that the region $|\x|\leq R(t)$ is invariant under the flow induced by $\u$.

To estimate $|\u \cdot \x|/|\x|$ we write $\u$ as given by the Biot-Savart law in terms of its curl, the {\it filtered} vorticity $\omega$:
\[\u = \u(\x,t) = \int \frac{(\x-\y)^{\perp}}{2\pi |\x - \y|^2} \omega(\y,t)\,\d\y.\]
We note that, although $\tilo(\cdot,t)$ is compactly supported at all times, $\omega(\cdot,t)$ is not. This is a consequence of $\tilo$ being
non-negative. However, $\omega = G^{\alpha}\ast \tilo$ and $G^{\alpha}$ is bounded and exponentially decaying at infinity. This implies that $\omega$ is bounded and belongs to $L^p(\real^2)$ for each $1\leq p \leq \infty$.

As we noted in section 2, the center of mass for both the filtered and unfiltered vorticities are the same, and they are conserved quantities; we may hence assume, without loss of generality, that the common center of mass is at the origin -- otherwise we perform a change of variables to translate the center of mass to the origin. In addition, the moment of inertia for both filtered and unfiltered vorticities is a conserved quantity, which we assume to be initially finite. Just as in [\cite{ISG1999}, equation (5)] we have
\[
\left| \u(\x,t) \cdot \frac{\x}{|\x|} \right| \leq \frac{C}{|\x|^3} + \int_{|\x-\y|<|\x|/2}\;\; \frac{\omega(\y,t)}{|\x-\y|} \,\d\y,
\]
where $C$ depends on the moment of inertia of $\omega$.

Lemma 2.1 of \cite{ISG1999} applies to $\omega$ also, so that we get
\begin{equation} \label{easyestimate}
\left| \u(\x,t) \cdot \frac{\x}{|\x|} \right| \leq \frac{C}{|\x|^3} + C \left(\int_{|\y|>|\x|/2}\; \omega(\y,t) \,\d\y\right)^{1/2},
\end{equation}
with $C$ depending on the $L^{\infty}$-norm of $\tilo_0$, along with the center of mass and the moment of inertia of $\tilo_0$.

The last, and most complicated, step of the proof consists in estimating the mass of filtered vorticity near infinity to show that it gives a negligible contribution for large $|\x|$. This is the content of Proposition 2.1 of \cite{ISG1999}, where it was shown that the mass of the Euler vorticity, near infinity, is smaller than any algebraically decaying function of $|\x|$. The idea is to smoothly approximate the mass of vorticity far from the center of motion and produce a differential inequality to which Gronwall's Lemma can be applied. We will use this idea, but for the mass of the unfiltered vorticity instead. We conclude by using \eqref{filtINTERMSofunfilt}, as the mass of unfiltered vorticity near infinity controls that of the filtered vorticity, up to an exponentially small error.

Following the same strategy as in \cite{ISG1999}, we introduce
\[f_r=f_r(t) = \int \phi_r(\y)\tilo(\y,t)\,\d\y,\]
where
\[\phi_r (\y) = \eta \left(\frac{|\y|^2 - r^2}{\lambda r^2}\right),\;\;\;\mbox{ and } \eta(s) = \frac{e^s}{e^s + 1} .\]
Here, $0 < \lambda = \lambda(r) < 1$ is a parameter which is chosen later.
Recall the main properties of $\eta$:
\begin{enumerate}
\item[(i)] $\eta$ is positive and monotone increasing,
\item[(ii)] $ 0 < \eta^{\prime}(s)\leq \min \{\eta (s) , e^{-|s|} \}$,
\item[(iii)] $|\eta^{\prime\prime}(s)|\leq \eta(s).$
\end{enumerate}
As in \cite{ISG1999} we will estimate $f_r$. We note, as do they, that
\[\frac{1}{2}\int_{|\y|>r}\tilo(\y,t)\,\d\y = \eta(0)\int_{|\y|>r}\tilo(\y,t)\,\d\y < f_r(t).\]
We differentiate $f_r$ in time to find:
\[f_r^\prime(t)=\int \nabla\phi_r(\y)\cdot\u(\y,t) \tilo(\y,t)\,\d\y \]
\[=\frac{1}{2\pi}\int\int \nabla\phi_r(\y)\cdot \frac{(\y-\x)^{\perp}}{|\y-\x|^2}\gamma_\alpha (|\y -\x|) \tilo(\x,t)\,\d\x\, \tilo(\y,t) \, \d\y\]
\[=\frac{1}{\pi}\int\int \eta^\prime\left(\frac{|\y|^2-r^2}{\lambda r^2}\right)  \frac{\y \cdot(\y-\x)^{\perp}}{\lambda r^2 |\x-\y|^2}\gamma_\alpha (|\x -\y|) \tilo(\x,t)\tilo(\y,t)\,\d\x \d\y.\]
Let $L=L(\x,\y)$ and $M=M(\x,\y)$ be defined as follows:
\begin{equation}\label{Lfunction}
L(\x,\y)\equiv \eta^\prime\left(\frac{|\y|^2-r^2}{\lambda r^2}\right)  \frac{\y \cdot(\y-\x)^{\perp}}{\lambda r^2}
\left(\frac{1}{ |\x-\y|^2} - \frac{1}{|\y|^2}\right) \gamma_\alpha (|\x -\y|) \tilo(\x,t)\tilo(\y,t)
\end{equation}
and
\begin{equation}\label{Mfunction}
M(\x,\y)\equiv  -\eta^\prime\left(\frac{|\y|^2-r^2}{\lambda r^2}\right)  \frac{\y \cdot \x^{\perp}}{\lambda r^2}
 \frac{1}{|\y|^2} \left[\, \gamma_\alpha (|\x -\y|) - \gamma_\alpha(|\y|) \,\right] \,\tilo(\x,t)\tilo(\y,t).
\end{equation}

Note that

\begin{eqnarray}
 \nonumber L(\x,\y)+M(\x,\y)= \eta^\prime\left(\frac{|\y|^2-r^2}{\lambda r^2}\right)  \frac{\y \cdot(\y-\x)^{\perp}}{\lambda r^2|\x-\y|^2}
\gamma_\alpha (|\x -\y|) \tilo(\x,t)\tilo(\y,t) \\
\label{LplusMfunction} + \eta^\prime\left(\frac{|\y|^2-r^2}{\lambda r^2}\right)  \frac{\y \cdot \x^{\perp}}{\lambda r^2}
 \frac{1}{|\y|^2}  \gamma_\alpha(|\y|) \,\tilo(\x,t)\tilo(\y,t).
\end{eqnarray}

Since the center of mass of $\tilo(\cdot,t)$ is assumed to vanish we have
\begin{equation}\label{frprimeLplusM}
  f_r^\prime(t)=\frac{1}{\pi} \int \int [L(\x,\y)+M(\x,\y) ]\, \d\x \d\y.
\end{equation}
Recall the decomposition of $\real^2$ used in \cite{ISG1999}:
\begin{eqnarray}
  \Omega_1 &=& \left\{ (\x,\y) \,\Big| \,|\y|<\frac{r}{2} \mbox{ or } |\y|>\frac{3r}{2} \right\} \label{Omega1} \\
  \Omega_2 &=& \left\{ (\x,\y) \,\Big| \, \frac{r}{2} \leq |\y|\leq \frac{3r}{2}, |\x| < \frac{r}{4} \mbox{ or } |\x|>\frac{7r}{4} \right\} \label{Omega2} \\
  \Omega_3 &=& \left\{ (\x,\y) \,\Big| \, \frac{r}{2} \leq |\y|\leq \frac{3r}{2}, \frac{r}{4}\leq |\x|\leq\frac{7r}{4} \right\} \label{Omega3} \\
  \Omega_4 &=& \left\{ (\x,\y) \,\Big| \, \frac{r}{4}< |\y| < \frac{7r}{4}, \frac{r}{4}\leq |\x| \leq \frac{7r}{4}  \right\}. \label{Omega4}
\end{eqnarray}
Then we have: $\real^2=\Omega_1 \cup \Omega_2 \cup \Omega_3$ and the unions are disjoint. In addition, since $\Omega_3 \subset \Omega_4$, we find $\Omega_3=\Omega_4 \setminus (\Omega_4 \setminus \Omega_3)$. Hence it follows that
\begin{eqnarray}
  \nonumber \pi f_r^\prime(t) &=& \int_{\Omega_1} [L(\x,\y)+M(\x,\y) ]\, \d\x \d\y + \int_{\Omega_2} [L(\x,\y)+M(\x,\y) ]\, \d\x \d\y \\ \label{LplusMdecomp} \\
 \nonumber  &+& \int_{\Omega_4} [L(\x,\y)+M(\x,\y) ]\, \d\x \d\y - \int_{\Omega_4\setminus\Omega_3} [L(\x,\y)+M(\x,\y) ]\, \d\x \d\y
\end{eqnarray}

We proceed to estimate each term in \eqref{LplusMdecomp}.

In what follows we will use repeatedly the bound
\begin{equation}\label{ZintermsofIandM}
  \int |\x|\tilo(\x,t)\,\d\x \leq \frac{m_0+i_0}{2}.
\end{equation}

We begin by noting that, on $\Omega_1$, we have
\[\left| \eta^\prime \left(\frac{|\y|^2-r^2}{\lambda r^2}\right) \right| \leq e^{-1/(2\lambda)}.\]
Hence, from the expression for $L+M$ \eqref{LplusMfunction}, we obtain
\begin{eqnarray}
\nonumber \int_{\Omega_1} |L(\x,\y)+M(\x,\y) |\, \d\x \d\y \leq & \ds{\frac{1}{e^{1/(2\lambda)}\lambda r^2} \int_{\Omega_1}  \left(\frac{\gamma_\alpha(|\x-\y|)}{|\x-\y|} \,|\y|\,\tilo(\y,t)\,\tilo(\x,t) \right.}\\
\nonumber \\
\nonumber & \ds{ \left. +\frac{\gamma_\alpha(|\y|)}{|\y|} \,|\x|\,\tilo(\x,t)\,\tilo(\y,t)\right)\, \d\y \d\x} \\
\nonumber \\
\leq & \ds{\frac{1}{e^{1/(2\lambda)}\lambda r^2} C\,m_0(m_0+i_0)}, \label{Omega1est}
\end{eqnarray}
where we have used, also, item 1 of Lemma \ref{gammaalphaprops},  that $\gamma_\alpha$, $\tilo \geq 0$ and \eqref{ZintermsofIandM}.

Next, as in \cite{ISG1999} we observe that $\Omega_4\setminus\Omega_3 \subset \Omega_1$, and, therefore, we also have
\begin{equation}\label{Omega4setminusOmega3est}
\int_{\Omega_4\setminus\Omega_3} |L(\x,\y)+M(\x,\y) |\, \d\x \d\y \leq  \frac{1}{e^{1/(2\lambda)}\lambda r^2} C\,m_0(m_0+i_0).
\end{equation}

Let us proceed with the analysis on $\Omega_2$. Here we will first estimate the integral of $L$ and then of $M$. From the definition of $L=L(\x,\y)$ we have:
\begin{eqnarray}
\nonumber \ds{\int_{\Omega_2} |L(\x,\y)|\, \d\x \d\y} & = \ds{\int_{\Omega_2}
\eta^\prime\left(\frac{|\y|^2-r^2}{\lambda r^2}\right)  \frac{|\y \cdot \x^{\perp}}{\lambda r^2}
 \frac{|\x\cdot(2\y-\x)|}{ |\x-\y|^2 |\y|^2}\gamma_\alpha (|\x -\y|) \tilo(\x,t)\tilo(\y,t) \, \d\x \d\y} \\
\nonumber \\
\nonumber & \leq \ds{\int_{\Omega_2} \eta \left(\frac{|\y|^2-r^2}{\lambda r^2}\right) \frac{|\y||\x|}{\lambda r^2}
\frac{|\x|\cdot 7|\x-\y|}{|\x-\y|^2|\y|^2} \gamma_\alpha(|\x-\y|)\tilo(\x,t)\tilo(\y,t)\,\d\x \d\y }\\
\nonumber \\
& \leq \ds{C\frac{i_0f_r(t)}{\lambda r^4}}. \label{Omega2estL}
\end{eqnarray}
We have used, above, property (ii) of $\eta$, together with the facts that $|2\y-\x|\leq 7|\x-\y|$ in $\Omega_2$, and $|\x-\y|$, $|\y| > Cr$ in $\Omega_2$. We also used \eqref{gammaalphabdd} and the definition of $f_r(t)$.

Next we estimate the integral of $M$ on $\Omega_2$. We find:
\begin{eqnarray}
\nonumber & \ds{ \int_{\Omega_2} |M(\x,\y)|\, \d\x \d\y } \\
\nonumber \\
\nonumber & \leq \ds{\int_{\Omega_2} \eta^\prime\left(\frac{|\y|^2-r^2}{\lambda r^2}\right)  \frac{|\y \cdot \x^{\perp}|}{\lambda r^2}
 \frac{1}{|\y|^2} \left[\, |\gamma_\alpha (|\x -\y|) -1| \right.}\\
 \nonumber \\
 \nonumber & \ds{\left.+|1- \gamma_\alpha(|\y||) \,\right] \,\tilo(\x,t)\tilo(\y,t) \, \d\x \d\y }\\
 \nonumber \\
& \leq \ds{ C\frac{(m_0+i_0)^2}{\lambda r^4}e^{-Cr}}. \label{Omega2estM}
\end{eqnarray}
Here we used item 2 of Lemma \ref{gammaalphaprops}, \eqref{ZintermsofIandM}, along with $|\x-\y|$, $|\y|>Cr$ in $\Omega_2$.

We now need to analyze $L+M$ in $\Omega_4$. This term is the most delicate because it corresponds to the near-field estimate, where $\x$ and $\y$ are close together. Let us start with the integral of $L$ in $\Omega_4$. Since $\Omega_4$ is symmetric with respect to $\x$ and $\y$ we have
\begin{equation} \label{Lsymmetrize}
\int_{\Omega_4} L(\x,\y) \, \d\x \d\y = \frac{1}{2} \int_{\Omega_4} [L(\x,\y) + L(\y,\x)]\, \d\x \d\y.
\end{equation}
We rewrite $L(\x,\y)+L(\y,\x)$ as below:
\[L(\x,\y)+L(\y,\x) = L_1(\x,\y)+L_2(\x,\y),
\]
where
\begin{eqnarray}
\nonumber
& L_1(\x,\y)= \ds{\frac{1}{\lambda r^2}
\left[
\eta^\prime \left(\frac{|\y|^2-r^2}{\lambda r^2} \right) \y\cdot (\x-\y)^\perp
-   \eta^\prime \left(\frac{|\x|^2-r^2}{\lambda r^2} \right)\x\cdot (\x-\y)^\perp
\right]} \\
& \times \ds{\left(
\frac{1}{|\x-\y|^2} - \frac{1}{|\y|^2}
\right)
\gamma_\alpha(|\x-\y|)\tilo(\x,t)\tilo(\y,t) } \label{L1function}
\end{eqnarray}
and
\begin{equation}\label{L2function}
L_2(\x,\y)= \frac{1}{\lambda r^2}
\eta^\prime \left(\frac{|\x|^2-r^2}{\lambda r^2} \right)\x\cdot (\x-\y)^\perp
\left(
\frac{1}{|\x|^2} - \frac{1}{|\y|^2}
\right)
\gamma_\alpha(|\x-\y|)\tilo(\x,t)\tilo(\y,t).
\end{equation}

Now,  $\y\cdot (\x-\y)^\perp  = \x\cdot (\x-\y)^\perp$ gives
\begin{eqnarray*}
L_1(\x,\y)=  \ds{\frac{1}{\lambda r^2}
\left[
\eta^\prime \left(\frac{|\y|^2-r^2}{\lambda r^2} \right)
-   \eta^\prime \left(\frac{|\x|^2-r^2}{\lambda r^2} \right)
\right]}\x\cdot (\x-\y)^\perp \\
\ds{\times \left( \frac{1}{|\x-\y|^2} - \frac{1}{|\y|^2} \right) \gamma_\alpha(|\x-\y|)\tilo(\x,t)\tilo(\y,t)}.
\end{eqnarray*}

If we follow the same estimates performed for expression (15) in \cite{ISG1999} we conclude that, in $\Omega_4$,
\begin{equation} \label{L1functionest}
|L_1(\x,\y)| \leq \frac{C}{\lambda^2r^2}\left[
\eta \left(\frac{|\y|^2-r^2}{\lambda r^2} \right)
+   \eta \left(\frac{|\x|^2-r^2}{\lambda r^2} \right)
\right]\gamma_\alpha(|\x-\y|)\tilo(\x,t)\tilo(\y,t).
\end{equation}
Therefore, since $0 < \gamma_\alpha \leq 1$,
\begin{equation} \label{Omega4estL1}
\int_{\Omega_4} |L_1(\x,\y)|\,\d\x \d\y \leq \frac{ Ci_0}{\lambda ^2 r^4}f_r(t).
\end{equation}
Similarly, the estimates in the expression (14) in \cite{ISG1999} lead to
\begin{equation} \label{Omega4estL2}
\int_{\Omega_4} |L_2(\x,\y)|\,\d\x \d\y \leq \frac{Ci_0 }{\lambda ^2 r^4}f_r(t).
\end{equation}

Lastly, we must analyze $M$ in $\Omega_4$. We have
\begin{equation}\label{MfunctionestA}
|M(\x,\y)| \leq \eta \left(\frac{|\y|^2-r^2}{\lambda r^2}\right)
 \frac{1}{\lambda r^2|\y|^2} |\y \cdot \x^{\perp}|\left|\, \gamma_\alpha (|\x -\y|) - \gamma_\alpha(|\y|) \,\right| \,\tilo(\x,t)\tilo(\y,t).
\end{equation}
Now
\[
|\y \cdot \x^{\perp}|\left|\, \gamma_\alpha (|\x -\y|) - \gamma_\alpha(|\y|) \,\right| =
|\y \cdot \x^{\perp}|\left|\, (\gamma_\alpha (|\x -\y|) - 1) +(1 - \gamma_\alpha(|\y|)) \,\right|
\]
\[
\leq |(\y - \x) \cdot \x^{\perp}||\gamma_\alpha (|\x -\y|) - 1| +
|\y \cdot \x^{\perp}|
|1 - \gamma_\alpha(|\y|)|
\]
\[
\leq |\x|\,\left[\,|\x - \y||\gamma_\alpha (|\x -\y|) - 1| +
|\y|
|1 - \gamma_\alpha(|\y|)|\, \right]
\]
\[
\leq C|\x|,
\]
by Lemma \ref{gammaalphaprops}, item 2.

Therefore, since $|\y| > r/4$,
\begin{equation}\label{MfunctionestB}
|M(\x,\y)| \leq \frac{C}{\lambda r^4} \eta \left(\frac{|\y|^2-r^2}{\lambda r^2}\right)
 |\x| \,\tilo(\x,t)\tilo(\y,t).
\end{equation}
It follows, using \eqref{ZintermsofIandM} and \eqref{MfunctionestB}, that
\begin{equation} \label{Omega4estM}
\int_{\Omega_4} |M(\x,\y)|\,\d\x \d\y \leq \frac{C(m_0+i_0)}{\lambda r^4} f_r(t).
\end{equation}

Putting together \eqref{Omega1est}, \eqref{Omega4setminusOmega3est}, \eqref{Omega2estL}, \eqref{Omega2estM}, \eqref{Omega4estL1}, \eqref{Omega4estL2} and \eqref{Omega4estM} leads to
\begin{equation}\label{differentialineq}
  \pi f_r^\prime (t) \leq \frac{C\,m_0(m_0+i_0)}{e^{1/(2\lambda)}\lambda r^2}  +
  \frac{C(m_0+i_0)^2}{\lambda r^4}e^{-Cr} + \left( \frac{Ci_0 }{\lambda ^2 r^4} + \frac{C(m_0+i_0)}{\lambda r^4}\right) f_r(t).
\end{equation}

Recall, from \cite{ISG1999}, that, if $r > 2R_0$, then
\begin{equation} \label{frofzero}
f_r(0) \leq \frac{m_0}{e^{1/(2\lambda)}}.
\end{equation}

Our choice of $\lambda$ will include the condition $\lambda < 1$.

Let
\[C_1 = \frac{1}{\pi}\max\{Cm_0(m_0+i_0), \, C(m_0+i_0)^2,\, 2Ci_0, \, 2C(m_0+i_0)\}.
\]

Then \eqref{differentialineq} becomes
\begin{equation}\label{differentialineqNew}
  f_r^\prime (t) \leq \frac{C_1}{\lambda^2 r^4} (\lambda r^2 e^{-1/(2\lambda)} + \lambda e^{-Cr}) + \frac{C_1}{\lambda^2 r^4} f_r(t).
\end{equation}

From \eqref{frofzero}, \eqref{differentialineqNew} and Gronwall's lemma we conclude that, for $r > 2R_0$,

\begin{equation} \label{froftestA}
f_r(t) \leq \left( m_0 + \lambda r^2 + \lambda \exp\left\{\frac{1}{2\lambda}  -Cr \right\}\right) \exp\left\{\frac{C_1t}{\lambda^2 r^4} - \frac{1}{2\lambda} \right\}.
\end{equation}

As in \cite{ISG1999} we choose, for each $k \in \mathbb{N}$,
\[\lambda = \left[ 4(k+2) \log \left(\frac{r}{R_0}\right) \right]^{-1}.\]

Then, if $t\leq \lambda r^4/(4C_1)$, for $\lambda$ above, we obtain
\[f_r(t) \leq  C(r) \frac{1}{r^k},\]
where
\[\hspace{-3cm} C(r)= m_0 \frac{R_0^{k+2}}{r^2}  + \frac{R_0^{k+2}}{4(k+2)\log(r/R_0)} \]
\[\hspace{2cm} + \frac{R_0^{k+2}}{4r^2(k+2)\log(r/R_0)}
\exp\left \{2(k+2)\log\left(\frac{r}{R_0}\right)-Cr\right \}.
\]
Clearly, if $r > 2R_0$ then $C(r)$ is bounded by, say, some constant $C_2>0$.

Following \cite{ISG1999} we find that, if $r > 2R_0 + C_3 [t\log(2+t)]^{1/4}$, for some $C_3$ sufficiently large, then
\[f_r(t)\leq \frac{C_3}{r^k}.\]

Choosing $k=6$ and putting this together with \eqref{easyestimate} and \eqref{filtINTERMSofunfilt}, noting that the exponentially decaying error in \eqref{filtINTERMSofunfilt} may also be estimated as $\mathcal{O}(r^{-6})$, leads to the conclusion, similarly to \cite{ISG1999}, that
$\{r\leq 8R_0 + C_3[t\log(2+t)]^{1/4}\}$ is invariant under the flow, as desired. This concludes the proof.

\end{proof}

\section{Conclusions}

As pointed out in the Introduction, the Euler-$\alpha$ system is closely related to the vortex blob method. This is a numerical method introduced by A. Chorin in \cite{Chorin1973}; for a precise description of the implementation of the vortex blob method and a state-of-the-art convergence result see \cite{XinLiu1995}. We will briefly recall the construction.

Choose a radially symmetric blob function $\phi \in C^{\infty}(\real^2)$ satisfying
\[
\phi \geq 0; \hspace{.2cm} \int_{\real^2}\phi\,dx=1; \hspace{.2cm} \int_{|x| \leq 1} \phi(x)\,dx\geq \frac{1}{2}.
\]

We will need to assume exponential decay of the blob function, i.e., that there exist  $C_1$, $C_2>0$  such that   $|\phi (x)|\leq C_1 e^{-C_2 x}$  for all $x\in\real^2$.

For each $\vare > 0$ set $\phi_\vare = \phi_\vare (x) = \vare^{-2}\phi(\vare^{-1}x)$. Denote $K_\vare = K_\vare(x) = K \ast \phi_\vare (x)$, the regularized Biot-Savart kernel.

Let $\omega_0 \in \mathcal{BM}_c(\real^2)$. Consider a discretization of the support of $\omega_0$ by non-overlapping squares $Q_i$ with sides of length $1/N$ and centers $X_{i,0}$, $i=1,\ldots,M$, where $M=M(N)$ is the number of squares. Write $m_i = \int_{Q_i} \omega_0 \,dx$. The vortex blob approximation consists in first solving the ODEs
\begin{equation} \label{vortblobode}
\left\{
\begin{array}{l}
 \ds{\frac{d}{dt}} X_i = \sum_{j=1}^M m_j K_\vare (X_i-X_j) \\
  X_i (t=0) = X_{i,0},
\end{array}
\right.
\end{equation}
i = 1,\ldots,M. Note that $X_i=X_i(t)$ also depend on $\vare$ and on $N$.

We then write
\[  \omega_\vare \equiv \sum_j m_j \delta_{X_j(t)};\]
\[\omega^\vare = \omega^\vare (x,t) \equiv \sum_j m_j \phi_\vare(x-X_j(t));\]
\[u^\vare = u^\vare (x,t) \equiv \sum_j m_j K_\vare (x-X_j(t)).\]

With this notation \eqref{vortblobode} becomes the PDE
\[\partial_t \omega_\vare + u^\vare\cdot\nabla\omega_\vare.\]
The analogy with the vorticity formulation of the Euler-$\alpha$ equations can be seen by substituting $u^\vare$ by $\u$ and $\omega_\vare$ by $\tilo$.

It is known and easy to verify that total mass, center of mass and moment of inertia are conserved for both $\omega_\vare$ and $\omega^\vare$, i.e.
\begin{eqnarray*}
\nonumber \int \, \d \omega_\vare (x,t) & = &\sum_i m_i \equiv C_1 \\
\nonumber \int \x \,\d \omega_\vare (x,t)  & = & \sum_i  m_i X_i(t) \equiv C_2 \\
\nonumber \int |\x|^2 \, \d \omega_\vare(x,t)  & = & \sum_i  m_i |X_i(t)|^2 \equiv C_3 \\
\nonumber \int \omega^\vare (\cdot,t)\, dx & = & \sum_i m_i \equiv C_1 \\
\nonumber \int \x \omega^\vare (\cdot,t)\, dx & = & \sum_i  m_i X_i(t) \equiv C_2 \\
\nonumber \int |\x|^2 \omega^\vare(\cdot,t)\,dx & = & \sum_i  m_i |X_i(t)|^2 + \sum_i m_i \vare^2 \int |x|^2\phi(x)\,dx \equiv C_3 + \vare^2 C_4.
\end{eqnarray*}

If the initial vorticity $\omega_0$ has distinguished sign and compact support, then the support of $\omega_\vare$ is contained in a fixed disk, for all time. Indeed, this is an immediate consequence of the conservation of (the discrete) moment of inertia; the approximations never leave a disk of fixed radius. However, this radius depends on the discretization of the initial vorticity and the estimate is lost as $N \to \infty$. A straightforward adaptation of Theorem \ref{mainthm} provides a confinement estimate for the vortex blob approximations of $\omega_0 \in \mathcal{BM}_{c,+}(\real^2)$, uniform with respect to $N$, for fixed $\vare$, and for any exponentially decaying blob function $\phi$. The most common blob function in the literature is the algebraically decaying {\it Krasny blob}, $\phi(x)= 1/[\pi(1+|x|^2))^2]$. The proof of Theorem \ref{mainthm} does not apply to the Krasny blob; exponential decay near infinity is needed to show that the mass of the vortex blob $\omega^\vare$  near infinity gives a negligible contribution to the transporting radial velocity $u^\vare \cdot x/|x|$, as discussed in section 4.

We close by mentioning a couple of open problems that arise from this work. First, it would be interesting to extend the confinement results to the Krasny blob approximation. Second, it is natural to seek a localization result, in the spirit of \cite{MP93} and \cite{CGM2013}, for the Euler-$\alpha$ equations.

\vspace{.5cm}

\footnotesize{
{\em Acknowledgments.} D.M.A. was supported in part by the National Science Foundation under Grant No. DMS-1515849.
M.C.L.F. was partially supported by CNPq grant \#306886/2014-6 and FAPERJ grant \#202.999/2017. H.J.N.L.'s research was supported in part by CNPq grant \#307918/2014-9 and FAPERJ grant \#202.950/2015. This material is based upon work supported by the National Science Foundation under Grant No. DMS-1439786 while the authors were in residence at the Institute for Computational and Experimental Research in Mathematics (ICERM) in Providence, RI, during the Spring 2017 semester.

}

\end{document}